\documentclass{article}

\usepackage{PRIMEarxiv}

\usepackage[utf8]{inputenc} 
\usepackage[T1]{fontenc}    
\usepackage{hyperref}       
\usepackage{url}            
\usepackage{booktabs}       
\usepackage{amsfonts}       
\usepackage{nicefrac}       
\usepackage{microtype}      
\usepackage{lipsum}
\usepackage{fancyhdr}       
\usepackage{graphicx, xcolor}       
\graphicspath{{media/}}     
\usepackage{amsfonts, amsthm, latexsym, amsmath, amscd, eucal, amssymb, epsf} 
\usepackage{mathtools, bbm}
\usepackage{comment}

\usepackage{tikz}
\usetikzlibrary{patterns}
\usepackage{pgfplots}

\newtheorem{theorem}{Theorem}

\newtheorem{corollary}[theorem]{Corollary}
\newtheorem{definition}[theorem]{Definition}
\newtheorem{lemma}[theorem]{Lemma}

\newtheorem{remark}[theorem]{Remark}
\def\uno{\mathbbm{1}}

\def\cA{\mathcal A}
\def\cB{\mathcal B}

\def\cE{\mathcal E}
\def\cF{\mathcal F}

\def\cQ{\mathcal Q}

\def\cT{\mathcal T}

\def\cW{\mathcal W}
\def\cU{\mathcal U}
\def\cX{\mathcal X}
\def\cY{\mathcal Y}
\def\cZ{\mathcal Z}

\def\bE{\mathbb E}

\def\bN{\mathbb N}

\def\bP{\mathbb P}

\def\bR{\mathbb R}

\def\me{\mathsf{e}}

\def\mE{\mathsf{E}}
\def\mG{\mathsf{G}}

\def\mv{\mathsf{v}}

\def\mx{\mathsf{x}}
\def\mz{\mathsf{z}}%

\title{Sticky Brownian motions on star graphs
}

\author{
Stefano Bonaccorsi \\
  Department of Mathematics \\
  Universit\`a degli Studi di Trento \\
  Italy \\
  \texttt{stefano.bonaccorsi@unitn.it} \\
   \And
  Mirko D'Ovidio \\
  Department of Basic and Applied Sciences for Engineering \\
  Sapienza University of Rome \\
  Italy\\
  \texttt{mirko.dovidio@uniroma1.it} \\
}

\begin{document}
\maketitle

\begin{abstract}
This paper is concerned with the construction of several stochastic processes in a star graph, that is a non-euclidean structure where some features of the classical modelling fail.
We propose a model for trapping phenomena with characterization of the traps in terms of a singular measure. This measure also defines a non-local operator by means of which we introduce a non-local dynamic condition for the parabolic problem on the star graph. We study semi-Markov processes on the rays of the graph in order to obtain a probabilistic representation of the motion on the whole graph. Extensions to general graph structures can be given by applying our results on star graphs.
\end{abstract}

\keywords{Traps \and Graphs \and Brownian excursions \and Non-local operators \and Dynamic boundary conditions.}

\tableofcontents

\section{Introduction}

The problem of a complete characterisation and construction of all possible Brownian motions on intervals was posed  by Feller
\cite{Feller1952, Feller1954a, Feller1954b} and later solved by It\^o and McKean \cite{Ito1963, Ito1974}.
Their solution is based on the theory of the local time for a Brownian motion and the theory of strong Markov processes.
More recently, several papers have addressed the question of what is a Brownian motion on a  graph.
In this paper, we continue in this line of research and we address the \emph{sticky} Brownian motion.
We shall work in the special case of a star graph, that is, the union of a finite number of copies of the half-line joined at the origin.
In graph terminology, we say that the star graph has a unique vertex (the origin) and $n$ edges of infinite length.
We shall formalise this construction in next subsection.
\\
The processes constructed on this space shares a lot of similarities with their correspondent one dimensional relatives.
For an example we discuss the elliptic problems associated with the infinitesimal generators of the processes (see also \cite{Freidlin1993, Freidlin2000} for related results).
In mathematical physics, the Laplace's problem $\Delta u = 0$ 
and the Poisson's problem $\Delta u = -g$ are related to the analysis of
steady states (for the first equation) and conservative fields (the latter).
Solutions to Laplace's problem are the {\em harmonic} functions on a domain with prescribed boundary conditions.
There is an important probabilistic representation for the solutions of both the Laplace's problem and the Poisson's problem, which we shall recover in our setting, see Theorems \ref{th.D} and \ref{th.P}.

Our choice of working in a star graph is not restrictive and can be justified as follows.
\\
Let $\mG$ be a generic graph with  finite numbers of edges and vertices.
In order to define a Brownian motion on $\mG$, we can proceed by defining the infinitesimal generator $A$ of the process, compare Theorem \ref{th:Kost},
by fixing boundary conditions in every vertex of the graph.
The resulting process is a Brownian motion, possibly with several different kinds of node conditions.
\\
Let us define for each vertex $\mv \in \mG$ the subgraph $\mG_\mv$ of all edges starting from $\mv$. This is further contained in the star graph with center $\mv$ having the same number of rays as the edges incident to $\mv$.
\\
A second approach to define the process on the whole graph is obtained by considering, first,
a family of  Brownian motions on the star graphs $\mG_\mv$, indexed by the set of vertices $\mv$; 
if the Brownian motion on the graph starts in a certain vertex $\mv$, it behaves like the Brownian motion defined on the star graph $\mG_\mv$ until the first time it reaches a different vertex.
From this moment, the behaviour of the Brownian motion on the graph coincides with that of the Brownian motion associated to this new vertex, again stopped at the first time of entering the set of neighbor vertices.
Compare for instance with \cite{Fitzsimmons2015}, where the authors introduce an embedded Markov chain to keep track
of the sequence of vertices in the graph visited by the Brownian motion.

Our results can be considered in many fields such as communications, social sciences, biology and others. 
Let us recall the trapping problems (see for example \cite{HollWeis}) and the well-known Bouchaud trap models (\cite{Bou92}). 
There is a vast literature on trapping problems for example in case of regular lattice and fractal structures. They involve a number of traps located in random locations, in these traps we have absorption. The Bouchaud trap models can be considered as reference models for trapping phenomena. They have the same scaling limit as a continuous time random walk leading to the fractional-kinetic equation and also in this case, there exists a consistent literature on the long-time behaviour of this models. In the present paper, we focus on sticky behaviours (including the absorption) realized through non-local dynamic conditions (including the so-called fractional kinetic equation). The non-local effects act independently on the motion determining holding times with infinite mean values, i.e., the process spends on the average an infinite amount of time  in the vertex of the star graph.

\vskip 1\baselineskip

Let us briefly sketch the structure of this work. In the remaining of this section, we provide all the main definitions and notation needed in the sequel, as well as a summary of our main results.
Section \ref{sez2} is devoted to present some background material, which extends to our framework known results about the construction of the Brownian motion on a star graph.
In Section \ref{sez:sticky} we define the sticky Brownian motion and we study its properties, with particular emphasis on its infinitesimal generator.
Section \ref{sez4} is devoted to the construction and analysis of the sticky Brownian motion with trapping star vertex.
In this case the associated infinitesimal generator is a nonlocal operator described through a fractional dynamic boundary condition.

\subsection{The star graph}

Let us consider a family of copies of the positive half lines $\cE = \{\me_j = [0, \infty),\ j = 1, \dots, n\}$.
Each point in $\cE$ is denoted formally by the couple $(j, x)$, where $j$ is the relevant ray considered and $x$ is the distance from the origin.

According to \cite{Mugnolo2019}, we introduce the equivalence relation on $\cE$ 
\begin{align*}
(j, x) \sim (k, y) \quad \Longleftrightarrow \quad \begin{cases} j = k\ \text{and}\ x = y \\ x = y = 0,\ \text {any}\ j, k. \end{cases}
\end{align*}

We define the \emph{star graph} as the quotient space $\mE = \cE/\sim$, i.e., we identify the starting points on all edges and
in $\mE$ the origin $0 \equiv (\cdot, 0)$ is the unique point that belongs to all the rays.

On every edge, we have an Euclidean structure given by the Euclidean distance, and a measure structure induced by the Lebesgue measure.
These structures are inherited by
the space $\mE$: it is a metric space with the distance
\begin{align*}
d((j,x), (k,y)) = |x - y| \uno_{j=k} + (x + y) \uno_{j \not= k}
\end{align*}
and a measure space with respect to the direct sum measure induced by the Lebesgue 
measure on every edge.

In particular, this metric-measure structure allows us to consider spaces of functions defined on the star graph $\mE$ based on topological and measure-theoretical notions:
in particular, we introduce the space $C_0(\mE)$
 of continuous functions $f \colon \mE \to \bR$ that vanish at infinity, equipped with the sup norm; and the Lebesgue spaces $L^p(\mE)$ with respect to the
Lebesgue measure.

Let $f \colon \mE \to \bR$.
As a shortcut we let $f_j(x) = f(j,x)$ for $j = 1, \dots, n$ and $x > 0$ and we define 
\begin{align}
f_j'(0) = \lim_{r \to 0} \frac{\partial}{\partial r} f(j,r).
\end{align}
Similarly, we let 
\begin{align*}
f_j''(x) =  \frac{\partial^2}{\partial x^2} f(j,x).
\end{align*}
We define $C^2_0(\mE)$ the space of functions in $C_0(\mE)$ that are twice continuously differentiable on each open ray $\mathring\me_j = (0,\infty)$
such that there exists finite the limit
\begin{align*}
f''(0) = \lim_{x \to 0} f_j''(x), \qquad j=1, \dots, n.
\end{align*}
Informally, we shall say that the second derivative $f''$ can be extended to a function in $C_0(\mE)$.

It shall be evident from the above that any random variable (and, therefore, stochastic process) $\cX$ with values in $\mE \setminus \{0\}$ is identified by two components, the \emph{spherical} component $\Theta$ which takes values in $\{1, \dots, n\}$, and the \emph{radial} component $X$, $X > 0$.
However, if $\cX = 0$, in order to uniquely define the spherical component, we impose $\Theta = 1$ and $X = 0$.

\subsection{Feller's Brownian motion on a star graph}

We adapt the following definition from \cite{Kostrykin2012a, Kostrykin2012b} to our setting.

\begin{definition}
\label{de:BM-Kost}
A Brownian motion $\cZ = \{\cZ_t,\ t \in [0,\infty)\}$ on $\mE$ is a diffusion process on $\mE$, such that {the radial component} $Z$ with absorption at 0 is equivalent to a Brownian motion on the half line $\bR_+$ with absorption at the origin.
\end{definition}

A diffusion process is a strong Markov process with continuous trajectories on $[0,\zeta)$, where $\zeta$ is its lifetime.
In \cite{Kostrykin2012a} the following characterization of a Brownian motion is stated.

\begin{theorem}
\label{th:Kost}
Assume that $\cZ$ is a Brownian motion on $\mE$ as defined in Definition \ref{de:BM-Kost}.
Then there exist constants $a, b, \{p_k\}, c \in [0,1]$, where $k=1, \dots, d$, with
\begin{align*}
\sum_{k=1}^n p_k = 1, \qquad a + b + c = 1, \qquad a \not= 1,
\end{align*}
such that the domain $D(A)$ of the generator $A$ of $\cZ$ in $C_0(\mE)$ consists exactly of those $f \in C^2_0(\mE)$ that satisfies
\begin{align}
\label{eq:gen-bc-bm}
a f(0) + \frac12 c f''(0) = b \, \sum_{k=1}^n p_k f_k'(0).
\end{align}
Moreover, for $f \in D(A)$,
\begin{align*}
A f({\color{black}j,}z) = \frac12 \frac{\partial^2}{\partial z^2}f({\color{black}j,}z).
\end{align*}
\end{theorem}

\begin{definition}
\label{def-Z}
A \emph{standard} Brownian motion $\cZ$ on $\mE$ is a Feller's Brownian motion with domain determined by $a = c = 0$, $b = 1$.
The values $\{p_k\}$ represent the probability of finding the Brownian motion on each of the edges $\me_k$.
\end{definition}

As occurs in the case of a real standard Brownian motion, giving the definition is not sufficient in order to show the existence of such a process.
The construction of this process will be provided in Section \ref{sez:constr}.

\subsection{Sticky Brownian motion}

According to Theorem \ref{th:Kost}, the definition of a \emph{sticky} Brownian motion can be given in terms of the parameters $a$, $b$, and $c$ in formula \eqref{eq:gen-bc-bm}.
Notice that actually there exists a family of sticky Brownian motions, depending on a parameter $\mu = \frac{b}{c} \in (0, \infty)$.

\begin{definition}
\label{de:SBM}
We say that a stochastic process $\cX$ is a $\mu$-sticky Brownian motion on a graph $\mE$ if it is a Brownian motion in the sense of Definition \ref{de:BM-Kost} and it satisfies $a = 0$, $b + c = 1$, $b = \mu c$, such that the domain of the infinitesimal generator $A$ of the process is given by \eqref{eq:gen-bc-bm}.
\end{definition}

The construction of a sticky Brownian motion on $\mE$ will be given in Section \ref{sez:sticky}. 
The idea, quite classical in this regards, is to define the process $\cX(t)$ via a suitable time change of the standard Brownian motion $\cZ(t)$, time change which employs the local time of $\cZ$.

Given a ball $B_r(0) = \{ \mx \in \mE\ :\ |x| < r\}$,
in Section \ref{sez.D}
we consider the problem of finding the distribution of place and time of exit from the ball.
It is possible to connect this problem with the analysis of harmonic functions on the graph (see Theorem \ref{th.D}) and the study of functions with prescribed second derivative (see Theorem \ref{th.P}).

Once we provide the existence of a sticky Brownian motion, 
this construction opens the way to define a further family of processes, by taking a further modification of the time change through a subordinator process as given in formula 
\eqref{eq:timechange-SBM} below.
Time changes  induce transformations of the speed of the motion of the process, thus leading to more complex and interesting dynamics.
A subordinator $H = \{H_t,\ t \ge 0\}$ is a non-decreasing L\'evy process of pure jump type, that is a process with stationary and independent increments, with trajectories that are c\`adl\`ag. The Laplace transform of $H_t$ is given by
\begin{align*}
\bE[e^{-s H_t}] = e^{-t \Phi(s)},
\end{align*}
where $\Phi(s)$ is a Bernstein function, i.e., there exist $\lambda > 0$ and a L\'evy measure $\phi$ on $\cB(\bR_+)$ such that
\begin{align*}
&\Phi(s) = \lambda s + \int_{(0,\infty)} ( 1 - e^{-\kappa s} ) \, \phi({\rm d}\kappa),
\\
&\int_{(0,\infty)} \min\{1, \kappa\} \, \phi({\rm d}\kappa) < \infty.
\end{align*}
We will assume that
\begin{equation}
t \mapsto H_t \text{ is strictly increasing, a.s.}
\end{equation}
which requires that $\phi((0,\infty)) = \infty$.
In this setting, the first passage time of the subordinator $H_t$ (i.e., the generalized right-inverse of $H_t$) is a process $L_t$
\begin{equation}
L_t = \inf\{s > 0\ :\ H_s > t\}.
\end{equation}
Since the sample paths of $H_t$ are a.s.\ strictly increasing, the process $L_t$ has a.s.\ continuous paths.

\section{The construction of the standard Brownian motion}
\label{sez2}

In their book \cite{Ikeda1981}, N.\ Ikeda and S.\ Watanabe provided the construction of a Brownian motion on the real line starting from
the collection of all excursions and then constructing the sample paths of the Brownian motion.
Here, we shall adapt their construction to our goal of defining a Brownian motion on the star graph.

\subsection{The space of positive excursions}
\label{sec:spexc}

Define $\cW_+$ the class of all continuous functions $w: [0,\infty) \to \bR_+$ with
\begin{itemize}
\item $w(0) = 0$, and
\item there exists $\sigma(w) > 0$ such that
\begin{itemize}
\item $w(t) > 0$, for $0 < t < \sigma(w)$, 
\item $w(t) = 0$ for $t \ge \sigma(w)$.
\end{itemize}
\end{itemize}
This space is called the \emph{space of positive excursions}. It may be endowed with the Borel $\sigma$-field $\sigma(\cW_+)$ generated by the cylindrical sets.

On the space $(\cW_+, \sigma(\cW_+))$ we define a $\sigma$-finite measure $n_+$ that satisfies
\begin{multline*}
n_+(\{w \in \cW_+\ :\ w(t_1) \in A_1, \dots, w(t_n) \in A_n\}) 
\\
= \int_{A_1} K(t_1, x_1) \, {\rm d}x_1 \int_{A_2} p_0(t_2 - t_1, x_1, x_2) \, {\rm d}x_2 \dots \int_{A_n} p_0(t_n - t_{n-1}, x_{n-1}, x_n) \, {\rm d}x_n
\end{multline*}
where
\begin{align*}
K(t,x) = \sqrt{\frac{2}{\pi t^3}} x \exp\left(-x^2/2t \right)
\end{align*}
is the density (in $t$) of the first passage time of the Brownian motion from level $x$, and 
\begin{align*}
p_0(t,x,y) = g_t(x-y) - g_t(x+y)
\end{align*}
is the density (in $y$) of the Brownian motion killed in $0$.

\begin{lemma}\label{le:mio1}
The measure of the set of excursions longer than $t$ is finite, and it holds
\begin{align*}
n_+(\{w \in \cW_+\ :\ \sigma(w) > t\}) = \sqrt{\frac{2}{\pi \, t}}.
\end{align*}
\end{lemma}

\begin{proof}
It is sufficient to compute
\begin{align*}
n_+(\{w \in \cW_+\ :\ \sigma(w) > t\}) = n_+(\{w \in \cW_+\ :\ w(t) > 0 \}) = \int_{0}^\infty K(t, x) \, {\rm d}x = \sqrt{\frac{2}{\pi \, t}}.
\end{align*}
\end{proof}

\subsection{Poisson random measures and Poisson point process}

Let $(\cW, \sigma(\cW))$ be a measurable space.
A \emph{Poisson random measure} $\mu$ on $(\cW, \sigma(\cW))$ is a collection of random variables $\{\mu(B),\ B \in \sigma(\cW)\}$
such that
\begin{itemize}
\item $\mu(\emptyset) = 0$ a.s.;
\item for each $B \in \sigma(\cW)$, $\mu(B)$ is Poisson distributed 
whenever $\mu(B) < \infty$.
\item if $\{B_k\}$ are disjoint elements in $\sigma(\cW)$, then $\mu( \cup B_k) = \sum \mu(B_k)$ a.s.;
\item if $\{B_k\}$ are disjoint elements in $\sigma(\cW)$, then $\{\mu(B_k)\}$ are independent random variables.
\end{itemize}

Notice that $\mu$ induces a $\sigma$-finite measure $\lambda$ on $(\cW, \sigma(\cW))$ by setting  $\lambda(B) = \bE[\mu(B)]$.
$\lambda$ is the \emph{intensity measure} associated with $\mu$.
The converse result also holds: given an intensity measure $\lambda$,
there exists a Poisson random measure with prescribed intensity measure; see \cite[Theorem I.9.1]{Ikeda1981}.

\begin{theorem}
\label{th:ikeda.I.9.1}
Given a $\sigma$-finite measure $n$ on $(\cW, \sigma(\cW))$, there exists a 
Poisson random measure $\mu$ on a probability space $(\Omega, \cF, \bP)$ such that $n(A) = \bE[\mu(A)]$ for all $A \in \sigma(\cW)$.
\end{theorem}

Next, we add time to the above construction.
Consider the space $S = [0,\infty) \times \cW$ endowed with the $\sigma$-algebra $\cA = \cB([0,\infty)) \otimes \sigma(\cW)$.
A \emph{Poisson point process} $p = (p(t),\ t \in [0,\infty))$ 
is an adapted process taking values in $\cW$ such that the random measure
\begin{align*}
N(t,A) := N([0,t) \times A) = \#\{0 \le s < t\ :\ p(s) \in A\}, \qquad \text{for $t \ge 0$, $A \in \sigma(\cW)$,}
\end{align*}
is a Poisson random measure on $S$.

We shall denote $D_p$ the countable support of the point measure $p$ (that we can interpret as the times of jump). Notice that $D_p$ is itself a random set.

The Poisson point process $p$ on $(\cW, \sigma(\cW))$ is \emph{stationary} if
its intensity measure $\bE[N({\rm d}t, {\rm d}x)]$ satisfies
\begin{align*}
\bE[N(t, A)] = t n(A)
\end{align*}
for some measure $n$ on $(\cW, \sigma(\cW))$.
The \emph{compensated Poisson random measure}
\begin{align*}
\tilde{N}(t, A) = N(t, A) - t n(A), \qquad \text{for $t \ge 0$, $A \in \sigma(\cW)$}
\end{align*}
is a martingale.

\subsection{The construction of the standard Brownian motion on $\mE$}
\label{sez:constr}

Let us begin with the \emph{excursion space} for $\cZ$ 
\begin{align*}
\cW_Z =  \{1, \dots, n\}  \times  \cW_+
\end{align*}
(recall the definition of $\cW_+$ in Section \ref{sec:spexc}).
In other words, the excursions of $Z$ are simply the excursions of the reflecting Brownian motion paired with the choice of a ray in $\mE$.

As we have seen before,
the \emph{excursion point process} is a Poisson point process with intensity measure given by the product of the Lebesgue measure on $[0,\infty)$ with a unique $\sigma$-finite measure $n$ on the excursion space.
In our setting, the excursion measure $n$ on $\cW_Z$ is given by the product measure $\mu \times n_+$, where $n_+$ is the excursion measure on $\cW_+$ and $\mu$ is a probability measure on $\{1, \dots, n\}$ with point masses $\{p_j\}$.
\\
Then, for any $U \subset \cW_+$ and $j \in \{1, \dots, n\}$ we have
\begin{align*}
n(\{j\} \times U) = p_j \, n_+(U).
\end{align*}

As a consequence of Theorem \ref{th:ikeda.I.9.1} we have

\begin{theorem}
There exists a stationary Poisson point process $p$ on $\cW_Z$ with intensity measure 
$n({\rm d}t, \{j\} \times U) = {\rm d}t \, p_j \, n_+(U)$.
\end{theorem}

The standard Brownian motion $\cZ$ on $\mE$ is constructed as follows. 
The radial component $Z(t)$ is a standard reflected Brownian motion, defined by the stationary Poisson point process $p_+$ on $\cW_+$ with intensity measure $n_+((0,t] \times U) = t \, n_+(U)$.

We introduce the increasing, right-continuous process associated with the reflected Brownian motion
\begin{align*}
A(s) = \int_0^{s^+} \int_{\cW_+} \sigma(x) \, N_+({\rm d}u, {\rm d}x) 
\end{align*}
where we recall that $\sigma(x) = \inf \{t > 0: x(t) = 0\}$ is the length of the excursion $x$, so that $A(s)$ counts the total length of the excursions touched by the Poisson point process $p$ on the time interval $[0,s]$.

Denote $\ell(t) = A^{-1}(t)$ the (pseudo) inverse of $A(t)$:
\begin{align*}
\ell(t) = A^{-1}(t) = \inf \{s \in \bar\bR\ :\ A(s) > t\} 
\end{align*}
with
\begin{align*}
A(\ell(t)^-) = \max\{A(s)\ :\ A(s) \le t\}
\end{align*}
Recall that $s \mapsto A(s)$ is strictly increasing, a.s.

We define a standard Brownian motion on $\mE$ as follows
\begin{align*}
\cZ(t) = p(u_0)(t) \qquad \text{if } 0 < t < A(u_0) = \sigma(p(u_0))
\end{align*}
and in general
\begin{align}
\label{eq:def-Z}
\cZ(t) =  p(\ell(t))(t - A(\ell(t)^-)).
\end{align}
If we denote $p(t) = (\alpha(t), p_+(t)) \in \{1, \dots, n\} \in \cW_+$; then we decompose the Brownian motion $\cZ(t)$ into the spherical and radial component as follows
\begin{align*}
Z(t) = p_+(\ell(t))(t - A(\ell(t)^-)), \qquad \Theta(t) = \alpha(t).
\end{align*}

\begin{corollary}
The above construction, in particular, implies that
$\ell^{\cZ}(t) = \ell(t)$ is the local time at the origin of the standard Brownian motion in $\mE$ and it coincides with the
local lime $\ell^+(t)$ of the one dimensional reflected Brownian motion.
\end{corollary}

Recall that the reflected Brownian motion $Z_t$ starting from $0$ satisfies the following properties:
\begin{enumerate}
\item the zero set $\zeta = \{t > 0\ :\ Z_t = 0\}$ has Lebesgue measure $0$;
\item the zero set $\zeta$ is dense in $(0,\varepsilon)$, for every $\varepsilon > 0$.
\end{enumerate}
Let us define $\tau_k^t$ the first time the process $\cZ$ enters the edge $k$ after time $t$, i.e.,
\begin{align*}
\tau_k^t = \inf\{s > t\ :\ \Theta_s = k\}.
\end{align*}
Then, given $\cZ(0) = 0$, it holds that 
\begin{align}
\text{for every $k = 1, \dots, n$, $\tau^0_k = 0$ almost surely.}
\end{align}

It is sufficient to prove that $\tau^0_k < \varepsilon$ for every $\varepsilon > 0$. But this follows because on $(0, \varepsilon)$ we
have an infinite number of returns to 0, hence for each $k$, $\{\Theta(t) = k\}$ occurs infinitely often in $(0, \varepsilon)$ almost surely.
For the arbitrariness of $\varepsilon$, we obtain the claim.

For $t > 0$, assume that $\Theta(t) \not= k$ (otherwise, by the continuity of trajectories, we have $\tau_k^t = t$ almost surely).
Let us denote $T^t_0 = \inf\{s > t\ :\ Z(s) = 0\}$.
Since $Z(t)$ has the same distribution of a reflected Brownian motion starting from 0, it follows that $\bP(Z(t) > 0) = 1$, and we can use the following formula obtained by the reflection principle
\begin{align*}
\bP(T^0_0 \in {\rm d}t \mid Z(0) = x) = \frac{x}{\sqrt{2 \pi t^3}} e^{-\frac{x^2}{2t}} \, {\rm d}t
\end{align*}
and strong Markov property to get
\begin{multline*}
\bP(T^t_0 \in {\rm d}s) = \int_0^\infty \bP^x(T^0_0 \in {\rm d}(s-t)) \, \bP(Z(t) \in {\rm d}x) 
\\
= \int_0^\infty \frac{x}{\sqrt{2 \pi (s - t)^3}} e^{-\frac{x^2}{2(s - t)}} \frac{2}{\sqrt{2 \pi t}} e^{-\frac{x^2}{2t}} \, {\rm d}x \, {\rm d}s 
= \frac{1}{\pi  s \sqrt{\frac{s-t}{t}}} \, {\rm d}s
\end{multline*}
or, equivalently, the first entrance time on the edge $k$ is
\begin{align*}
\bP(\tau^t_k \in {\rm d}s) = \left( \delta_t(s) \, \uno_{\{\Theta(t) = t\}} + \frac{1}{\pi  s \sqrt{\frac{s-t}{t}}}  \, \uno_{\{\Theta(t) \not = k\}} \right) \, {\rm d}s.
\end{align*}

As a consequence, we see that the restriction of the process $\cZ(t)$ to the edge $k$, defined as
\begin{align*}
Z_k(t) = \{0\} \, \uno_{\{\Theta(t) \not= k\}} + Z(t) \, \uno_{\{\Theta(t) = k\}}
\end{align*}
is not a strong Markov process, since the waiting time to leave the origin, given $X_k(t) = 0$, depends on the past history of the process.

\subsection{Infinitesimal generator}

In this section, following \cite{Kostrykin2012a}, we prove that the standard Brownian motion defined in \eqref{eq:def-Z}
has infinitesimal generator that satisfies
\eqref{eq:gen-bc-bm} with $a = c = 0$ as required by Definition \ref{def-Z}.

We shall denote $\tau_\varepsilon(\mz)$ the exit time from the ball of radius $\varepsilon$ around $\mz \in \mE$:
\begin{align*}
\tau_\varepsilon(\mz) = \inf\{t > 0\ :\ d(Z_t, \mz) \ge \varepsilon\}
\end{align*}

Let $f$ be a smooth function $f \colon \mE \to \bR$. 
Following \cite[Theorem 19.23]{Kallenberg2002},  we have 
\begin{align*}
A f(\mz) = \lim_{\varepsilon \downarrow 0} \frac{\bE[f(\cZ_{\tau_\varepsilon(\mz)}) - f(\mz) \mid \cZ_0 = \mz]}{\bE^\mz[\tau_\varepsilon(\mz)]}.
\end{align*}

Assume first that $\mz = (j,r)$ for $r > 0$, and suppose $\varepsilon < r$. Then $Z \sim B$ behaves like a one dimensional Brownian motion for small times, and we 
can use the following result about the exit time from a ball of the Brownian motion:
\begin{align*}
\bE[\tau_\varepsilon(r)] = \varepsilon^2
\end{align*}
(use the fact that $B^2_t - t$ is a martingale) and the symmetry of trajectories of the Brownian motion to get
\begin{align*}
Af(\mz) = \lim_{\varepsilon \downarrow 0} \frac12 \frac{f(j, r + \varepsilon) - 2 f(j, r) + f(j, r - \varepsilon)]}{\varepsilon^2} = \frac12 f_j''(r).
\end{align*}

Next we see what happens in the origin. First we state a result about the exit time from the ball that is proved in \cite[Lemma 2.1]{Kostrykin2012a}.

\begin{lemma}
For $\mz = 0$ it holds
\begin{align*}
{\bE^0[\tau_\varepsilon(0)]} = \varepsilon^2.
\end{align*}
\end{lemma}

Next, let $f \in C^2_0(\mE)$ with $f \in D(A)$. This entails the existence of the limit
\begin{align*}
A f(0) = \lim_{\varepsilon \downarrow 0} \frac{\bE^0[f(Z_{\tau_\varepsilon(0)}) - f(0)]}{\bE^0[\tau_\varepsilon(0)]}.
\end{align*}
As we have seen before, this is equal to
\begin{align*}
A f(0) = \lim_{\varepsilon \downarrow 0} \sum_{k=1}^d p_k \frac{f(k, \varepsilon) - f(0)}{\varepsilon^2}
\end{align*}
and by formally applying Taylor's formula we obtain
\begin{align*}
A f(0) = \lim_{\varepsilon \downarrow 0} \sum_{k=1}^d p_k \frac{\varepsilon f'_k(0) + \frac12 \varepsilon^2 f''_k(0)}{\varepsilon^2}
=  \frac12  f''(0) + \lim_{\varepsilon \downarrow 0} \frac{1}{\varepsilon}  \sum_{k=1}^d p_k f'_k(0) 
\end{align*}
therefore the limit exists and it is equal to $A f(0) = \frac12 f''(0)$ if and only if the Kirchhoff boundary condition holds.


\section{Sticky Brownian motion on $\mE$}
\label{sez:sticky}

In the first part of this section, we collect some properties of the sticky Brownian motion $\cX$ as defined in Definition \ref{de:SBM}.
We shall denote $T_0$ the first passage time from the origin $0 \in \mE$.
The first property stated in Definition \ref{de:BM-Kost} can be equivalently written as
\begin{align*}
\bE^{(i,x)}[f_1(\cX(t_1 \wedge T_0)) \dots f_k(\cX(t_k \wedge T_0))]
= \bE^x[f_1(i, X(t_1 \wedge T_0)) \dots f_k(i, X(t_k \wedge T_0))]
\end{align*}
for 
$k \in \bN$, $f_1, \dots, f_k \in C_0(\mE)$ and $0 \le t_1 < \dots < t_k < T_0$.

Since $\cX$ is a diffusion process, it is uniquely determined by either its generator $(A_\cX, D(A_\cX))$, 
the transition semigroup
\begin{align*}
\cQ_t f(\mx) = \bE^{\mx}[f(\cX(t))], \qquad t \ge 0, \quad \mx \in \mE, \quad f \in C_0(\mE)
\end{align*}
or the associated resolvent
\begin{align*}
\cU_\lambda f(\mx) = \int_0^\infty e^{-\lambda t} \cQ_t f(\mx) \, {\rm d}t, \qquad \lambda  \ge 0, \quad \mx \in \mE, \quad f \in C_0(\mE).
\end{align*}
Recall that the infinitesimal generator $(A_\cX, D(A_\cX))$ is, by Definition \ref{de:SBM},
\begin{align*}
A_\cX f(\mx) &= f_j''(x), \qquad \mx = (j, x) \\
D(A_{\cX}) &= \left\{f \in C_0^2(\mE) \ :\ \frac12 c f''(0) = b \sum_{k=1}^n p_k f'_k(0) \right\}.
\end{align*}

Now we characterize the transition semigroup
\begin{align*}
\cQ_t f ( \mx ) 
= & Q^D_t f_i(x) + \int_0^t  \cQ_{t-s} f(0)\, \mathbf{P}^x(T_0 \in {\rm d}s)
\end{align*}
where 
$Q^D_tf(x) = \mathbf{E}_x[f(X^D_t)]$ is the transition semigroup for the Brownian motion killed at the origin (Dirichlet semigroup) and the first passage time law depends only on the radial component $x$ of the starting point $\mx$, hence it is known to be \cite[page 107]{Revuz1999}
 $\displaystyle \mathbf{P}^x(T_0 \in {\rm d}s) = \frac{x}{s} g_s(x)$, and substituting in previous formula we have
\begin{align}
\label{eq.tr.sg.1}
\cQ_tf(\mx) 
= & Q^D_t f_i(x) + \int_0^t  \frac{x}{s} g_s(x) \cQ_{t-s} f(0)\, {\rm d}s,
\end{align}
which means that the knowledge of $\cQ_tf(0) = \bE^0[f(\cX_t)]$ is sufficient to determine the whole semigroup.
\\
Next, for any $t > 0$ it holds that $\cQ_tf$ belongs to $D(A)$, which implies that
\begin{align}
\label{eq:ident1}
\frac12 c (\cQ_t f)''(0) = b \sum_{k=1}^n p_k \left( \cQ_tf \right)'(k, 0).
\end{align}
Notice that in the left-hand side we have used the continuity of the second derivative in order to simplify the notation.

\medskip

Finally, we consider the resolvent operator. 
Let us compute the Laplace transform of \eqref{eq.tr.sg.1} to get
\begin{align}\label{eq.res.1}
\cU_\lambda f(\mx) = U_\lambda^D f_i(x) + \left( \underbrace{ \int_0^\infty e^{-\lambda t} \frac{x}{t} g_t(x) \, {\rm d}t }_{e^{-\sqrt{2\lambda } x} } \right) \cU_\lambda f(0)
\end{align}
where the resolvent operator of the killed Brownian motion is 
\begin{align*}
U^D_\lambda \varphi(x) = \int_0^\infty e^{-\lambda t} \int_{\bR_+} [g_t(x-y) - g_t(x+y)] \varphi(y) \, {\rm d}y \, {\rm d}t
= \frac{1}{\sqrt{2\lambda}} \int_{\bR_+} \left[ e^{-|x-y| \sqrt{2 \lambda}} - e^{-(x + y) \sqrt{2\lambda}} \right] \, \varphi(y) \, {\rm d}y.
\end{align*}

Taking the Laplace transform in both sides of \eqref{eq:ident1}, we aim to obtain the analog boundary condition for the resolvent operator $\cU_\lambda(\mx)$
\begin{align}
\label{eq:ident2}
c \int_0^\infty e^{-\lambda t} \partial_t \cQ_t f(0) \, {\rm d}t = b \sum_{k=1}^n p_k \partial_x \int_0^\infty e^{-\lambda t} \left( \cQ_tf \right)(k, x) \, {\rm d}t\Big|_{x=0}.
\end{align}
Next result provides a more tractable form of previous formula. 
Notice that in the left had side, we have uses the diffusion equation satisfied pointwise from the transition semigroup in order to compute a first order time derivative instead of the second order space derivative.

\begin{lemma}
\label{lemma10}
The resolvent operator of the sticky Brownian motion is completely determined by the following identities:
\begin{align}
\label{eq:ident3}
\begin{aligned}
\cU_\lambda f(\mx) &= U_\lambda^D f_i(x) + {e^{-\sqrt{2\lambda } x} }  \cU_\lambda f(0), \qquad \mx = (i, x)
\\
\left( \lambda + \frac{b}{c} \sqrt{2\lambda} \right) \cU_\lambda f(0)  &= f(0) + \frac{2b}{c}
 \sum_{k=1}^n p_k \hat{f_k}(\sqrt{2\lambda}). 
\end{aligned}
\end{align}
\end{lemma}

\begin{proof}
We can treat the left hand side of \eqref{eq:ident2} without computing the transition semigroup (by an integration by parts):
\begin{align*}
c \int_0^\infty e^{-\lambda t} \partial_t \cQ_t f(0) \, {\rm d}t = c \left( - f(0) + \lambda \cU_\lambda f(0) \right).
\end{align*}
The right hand side requires some more efforts.
First, recalling \eqref{eq.res.1}, we compute the space derivative of the resolvent operator of the killed Brownian motion
\begin{align*}
(U^D_\lambda \varphi)'(x) 
&= \left[ - \int_{(0,x)} e^{-(x-y) \sqrt{2 \lambda}} \varphi(y) \, {\rm d}y + \int_{(x,\infty)} e^{-(y-x) \sqrt{2 \lambda}} \varphi(y) \, {\rm d}y + \int_{(0,\infty)} e^{-(x + y) \sqrt{2\lambda}} \varphi(y) \, {\rm d}y\right]
\\
\lim_{x \downarrow 0} (U^D_\lambda \varphi)'(x) &= 2 \int_{(0,\infty)} e^{-\sqrt{2 \lambda} y} \varphi(y) \, {\rm d}y.
\end{align*}
It follows that
\begin{align*}
\lim_{x \downarrow 0} (\cU_\lambda f)'(i,x) = 2 \int_{(0,\infty)} e^{-\sqrt{2 \lambda} y} f_i(y) \, {\rm d}y -\sqrt{2\lambda} \cU_\lambda f_i(0)
\end{align*}
if we substitute in \eqref{eq:ident2} we obtain the thesis
\begin{align*}
\sum_{k=1}^n p_k \partial_x  \left( \cU_\lambda f \right)(k, x) \Big|_{x=0}
= -\sqrt{2\lambda} \cU_\lambda f(0)
+ 2 \sum_{i=1}^n p_i \int_{(0,\infty)} e^{-\sqrt{2 \lambda} y} f_i(y) \, {\rm d}y.
\end{align*}
\end{proof}

\subsection{The construction of a sticky Brownian motion}

Let us define
\begin{align}
\label{eq:timechange-SBM}
V(t) = t + \mu \,  \ell^\cZ(t), \qquad t \ge 0,
\end{align}
where $\ell^\cZ$ is the local time at $0 \in \mE$ of the standard Brownian motion on $\cZ$ on $\mE$.

In the following result, we prove that the process $\cX(t) = \cZ(V^{-1}(t))$ has the same infinitesimal generator of a sticky Brownian motion on $\mE$.
By construction, the inverse $V^{-1}(t)$ is a strictly increasing function that remains bounded by $t$, i.e., it slows down the reflecting Brownian motion $Z$ at the origin. Thus,
$\cX$ is forced to stop for a random amount of time at the origin.

\begin{theorem}
The process $\cX(t \wedge T_0)$ is equivalent in law to a Brownian motion absorbed at the origin for any starting point $\mx = (i, x) \not = 0$.
\end{theorem}

\begin{proof}
By construction,
since $\cX(t) = ( \Theta(t), X(t) )$ is the polar representation of the process $\cX$, with $\Theta(t)$ being the selected ray at time $t$ and $X(t)$ the corresponding radial component, it follows that $T_0$ is also the first passage time from 0 of the diffusion process $X$ on the half-line $\bR_+$.

In order to get the proof, we consider the resolvent operator
\begin{align*}
\bE^x \left[ \int_0^\infty e^{-\lambda t} X(t) \uno_{(t < T_0)} \, {\rm d}t \right]
=
\bE^x \left[ \int_0^\infty e^{-\lambda t} Z(V^{-1}(t)) \uno_{(t < T_0)} \, {\rm d}t \right]
=
\bE^x \left[ \int_0^\infty e^{-\lambda V(t)} Z(t) \uno_{(V(t) < T_0)} \, {\rm d}V(t) \right]
\end{align*}
and we observe that
\begin{align*}
\bP^x \left( \uno_{(t < T_0)} \ell^\cZ(t) = 0 \right) = 1
\end{align*}
or equivalently, we can take $V(t) = t$ in the last integral above, which implies, in particular,
\begin{align*}
\bE^x \left[ \int_0^\infty e^{-\lambda t} X(t) \uno_{(t < T_0)} \, {\rm d}t \right]
=
\bE^x \left[ \int_0^\infty e^{-\lambda t} Z(t) \uno_{(t < T_0)} \, {\rm d}t \right]
\end{align*}
hence the radial component $X$ behaves like a standard Brownian motion on $(t < T_0)$, as required.
\end{proof}

\begin{theorem}
The process $\cX(t) = \cZ(V^{-1}(t))$ is a sticky Brownian motion according to Definition \ref{de:SBM}.
\end{theorem}

\begin{proof}
We shall provide the thesis by proving that  the process $\cX(t) = \cZ(V^{-1}(t))$ has resolvent operator that satisfies the equation \eqref{eq:ident3}.
Let us compute
\begin{align*}
\cU_\lambda f(\mx) &= \bE^{\mx} \left[ \int_0^\infty e^{-\lambda t} f(\cZ(V^{-1}(t))) \, {\rm d}t \right]
\\
&= \bE^{\mx} \left[ \int_0^\infty e^{-\lambda (t + \mu \ell^\cZ(t))} f(\cZ(t)) \, {\rm d}(t + \mu \ell^\cZ(t)) \right]
= \cU^1_\lambda f(\mx) + \cU^2_\lambda f(\mx).
\end{align*}

Let us start from
\begin{align*}
\cU^1_\lambda f(\mx) =& \bE^{\mx} \left[ \int_0^\infty e^{-\lambda (t + \mu \ell^\cZ(t))} f(\cZ(t)) \, {\rm d}t \right]
\\
=&  \bE^{\mx} \left[ \int_0^{T_0} e^{-\lambda (t + \mu \ell^\cZ(t))} f(\cZ(t)) \, {\rm d}t \right] +  \bE^{\mx} \left[ \int_{T_0}^\infty e^{-\lambda (t + \mu \ell^\cZ(t))} f(\cZ(t)) \, {\rm d}t \right]
\\
\intertext{where $T_0$ is the first passage time from the vertex 0 for the process $\cX$, but it coincides with the first passage time from the vertex 0 for the standard Brownian motion $\cZ$, and it further coincides with the lifetime of a killed  Brownian motion $W$ on $\bR_+$}
=&  \bE^x \left[ \int_0^\infty e^{-\lambda t} f_i(W(t)) \, {\rm d}t \right] + \bE^{\mx} \left[ e^{-\lambda T_0} \bE^0 \left[ \int_0^\infty e^{-\lambda t - \lambda \mu \ell^\cZ(t)} f(\cZ(t)) \, {\rm d}t \right] \right]
\\
\intertext{we identify the first term with the resolvent operator of the killed Brownian motion $U^D_\lambda f_i(x)$ and we decompose the second term according to the assigned probability distribution on the edges
}
=& U^D_\lambda f_i(x) +  \bE^x \left[ e^{-\lambda T_0} \right] \bE^0 \left[ \sum_{k=1}^n p_k \int_0^\infty e^{-\lambda t - \lambda \mu \ell^Z(t)} f_j(Z(t)) \, {\rm d}t \right] 
\\
\intertext{we employ the known joint distribution of $Z(t)$ and $\ell^Z(t)$, compare \cite[page 45]{Ito1974}
$\bP^0(Z(t) \in {\rm d}y,\ \ell^Z(t) \in {\rm d}\omega) = 2 \frac{y+\omega}{t} g_t(y+\omega)$, and we use twice Fubini's theorem to get}
=& U^D_\lambda f_i(x) + e^{-\sqrt{2\lambda} x} \sum_{k=1}^n 2 p_k \int_{(0,\infty)} \int_{(0,\infty)} e^{ - \lambda \mu \omega} f_k(y) e^{- \sqrt{2\lambda} (y + \omega)} \, {\rm d}\omega \, {\rm d}y
\\
\cU^1_\lambda f(\mx) =& U^D_\lambda f_i(x) +  \frac{2}{\lambda \mu + \sqrt{2\lambda}} e^{-\sqrt{2\lambda} x} \sum_{k=1}^n p_k \hat{f_k}(\sqrt{2\lambda}). 
\end{align*}
Next
\begin{align*}
\cU^2_\lambda f(\mx) =& \bE^{\mx} \left[ \int_0^\infty e^{-\lambda (t + \mu \ell^\cZ(t))} f(\cZ(t)) \, {\rm d}(\mu \ell^\cZ(t)) \right] 
\\
=& - \frac{1}{\lambda} \bE^{\mx} \left[ \int_0^\infty e^{-\lambda t} f(0) \, {\rm d}(e^{- \lambda \mu \ell^\cZ(t)}) \right] 
\\
=& - \frac{f(0)}{\lambda} \bE^{\mx} \left[ \left.e^{-\lambda t - \lambda \mu \ell^\cZ(t)}\right|_{t=0}^{\infty} + \lambda 
 \int_0^\infty e^{-\lambda t - \lambda \mu \ell^\cZ(t)} \, {\rm d}t \right] 
\\
=&  f(0) \bE^{x} \left[ \frac{1}{\lambda}  -  
 \int_0^\infty e^{-\lambda t - \lambda \mu \ell^Z(t)} \, {\rm d}t \right] 
\end{align*}
By taking $\mx = 0$ we obtain
\begin{align*}
\cU^2_\lambda f(0) =& f(0) \left[ \frac{1}{\lambda}  -  2
 \int_0^\infty  \int_0^\infty  e^{-\lambda t - \lambda \mu \omega } g_t(\omega) \, {\rm d}t \, {\rm d}\omega \right] 
\\
=& f(0) \left[ \frac{1}{\lambda}  -  \frac{2}{\sqrt{2\lambda}}
 \int_0^\infty    e^{-\sqrt{2\lambda} \omega - \lambda \mu \omega } \, {\rm d}\omega \right] = f(0) \left[ \frac{1}{\lambda}  -  \frac{2}{\sqrt{2\lambda}} \frac{1}{\sqrt{2\lambda} + \lambda \mu  } \right]
\end{align*}
and putting the above computation together we get
\begin{align*}
\cU_\lambda f(0) 
= \frac{2}{\lambda \mu + \sqrt{2\lambda}}  \sum_{k=1}^n p_k \hat{f_k}(\sqrt{2\lambda}) + \frac{\mu}{\lambda \mu + \sqrt{2\lambda}} f(0)
\\
\mu \left(\lambda + \frac1\mu \sqrt{2\lambda} \right) \cU_\lambda f(0) 
= 2  \sum_{k=1}^n p_k \hat{f_k}(\sqrt{2\lambda}) + \mu f(0)
\\
 \left(\lambda + \frac1\mu \sqrt{2\lambda} \right) \cU_\lambda f(0) 
= 2 \frac{1}{\mu}  \sum_{k=1}^n p_k \hat{f_k}(\sqrt{2\lambda}) +  f(0)
\end{align*}
which coincides with \eqref{eq:ident3} when we take $\mu = \frac{c}{b}$.
\end{proof}

\subsection{Elliptic problems associated to the sticky Brownian motion}
\label{sez.D}

in this section we consider two classical problems associated with a diffusion operator, namely the Dirichlet problem and the Poisson problem;
they are naturally associated with the exit probabilities and the mean exit time, respectively.

\subsubsection{Dirichlet problem}

Recall that the infinitesimal generator $(A_\cX, D(A_\cX))$ is, by Definition \ref{de:SBM},
\begin{align*}
A_\cX f(\mx) &= f_j''(x), \qquad \mx = (j, x) \\
D(A_{\cX}) &= \left\{f \in C_0^2(\mE) \ :\ \frac12 c f''(0) = b \sum_{k=1}^n p_k f'_k(0) \right\}.
\end{align*}
Let $r > 0$ and $B_r = B_r(0)$ the open ball centred at the origin $0 \in \mE$ and defined as $B_r = \mathcal{B} / \sim$ with $\mathcal{B}=\{\mathbf{b}_j=[0, r), \, j=1, \ldots, n\}$. We notice that $\partial B_0(r) = \{ (e, r),\ e = 1, \dots, n\}$ consists of exactly $n$ points.
We introduce the first passage times
\begin{align*}
    T_{(e,r)} = \inf \{t > 0\ :\ \cZ(t) = (e,r) \}
\end{align*}
and the exit time from the ball
\begin{align*}
    T^\star = \inf \{t > 0\ :\ |\cZ(t)| = r \} = \min \{ T_{(e,r)},\ :\ e=1, \dots, n\}.
\end{align*}

The Dirichlet problem associated with $A_\cX$ is
\begin{align*}
\tag{D}\label{e.D}
\begin{cases}
\text{find $u \in D(A_\cX)$ such that}\\
A_\cX u(\mx) = 0, \qquad \mx \in B_r \\
u((j,r)) = \alpha_j, \qquad j = 1, \dots, n.
\end{cases}
\end{align*}

Let $e \in \{1, \dots, n\}$ be fixed and define the the function
\begin{align}\label{e.def.u}
    u(\mx) = \bP( T^\star = T_{(e,r)} \mid \cZ(0) = \mx)
\end{align}
that is the probability that the first exit from the ball occurs along the edge $e$.
Similar to the case of a real valued Brownian motion, we prove here that
$u(t,\mx)$ is an harmonic function for the infinitesimal generator $A_\cX$ of the sticky Brownian motion $\cX$.

\begin{theorem}\label{th.D}
    The function $u(\mx)$ defined in \eqref{e.def.u} satisfies the Dirichlet problem \eqref{e.D} on the ball $B_0(r)$
    with boundary conditions
    \begin{align*}
        u_e(r) = 1, \qquad u_j(r) = 0 \quad j \not= e.
    \end{align*}
\end{theorem}

\begin{proof}
    Let us first notice that the boundary conditions are obviously satisfied.
    \\
    It remains to prove that the identity $A_\cX u(\mx) = 0$ is satisfied and that $u \in D(A_\cX)$. 
    Maybe not so surprisingly, most of the work is concerned with this last condition.
    \\
    Suppose for simplicity that $\mx = (j,x)$ with $|x| > 0$; then for small $h$ we have $B_\mx(h) = \{(j, y),\ |y - x| < h\}$, and the strong Markov property of $\cX$ implies
    \begin{align*}
    u_j(x) =& \bP( T^\star = T_{(e,r)} \mid \cX(T_{B_\mx(h)}) = (j,x-h)) \bP( \cX(T_{B_\mx(h)}) = (j, x-h)) 
    \\ 
    &+ \bP( T^\star = T_{(e,r)} \mid \cX(T_{B_\mx(h)}) = (j, x+h)) \bP( \cX(T_{B_\mx(h)})  = (j, x+h))
    \\
    =& \frac12 u_j(x-h) + \frac12 u_j(x+h)
    \end{align*}
    hence, by Schwarz's theorem
    \footnote{compare, for instance, \cite[page 137]{Chung}:
    \begin{theorem} 
	Let $f$ be a continuous function defined in an interval $(a,b)$ such that the generalised second derivative is well defined, for each $x \in (a,b)$ and $h > 0$ such that $(x-h, x+h) \subset (a,b)$, and satisfies
	\begin{align*}
	\lim_{h \to 0} \frac{f(x+h) - 2 f(x) + f(x-h)}{h^2} = \varphi(x)
	\end{align*}
	for a continuous function $\varphi(x)$ defined in $(a, b)$. 
	Then $f$ is two times continuously differentiable  and it holds $f''(x) = \varphi(x)$ for each $x \in (a,b)$.
	\end{theorem}
    }
    \begin{align*}
	u_j''(x) = 0, \qquad x \in (0,r).
    \end{align*}
    In particular, it is
    $u_j(x) = \alpha_j^e x + \beta_j^e$ for every $(j,x)$ with $x > 0$, and the first equation in \eqref{e.D} is satisfied.
    \\
    Now we prove continuity of $u$ in $0$ (since the function is linear for $|\mx| > 0$, the continuity is obvious).
    We have the following representation of $u(\mx)$: if $\mx = (e, x)$ then
    \begin{align*}
        u_e(x) =& \bP( T_r < T_0 \mid B(0) = x) + \bP( T_0 < T_r \mid B(0) = x) \bP( T^\star = T_{(e,x)} \mid \cZ(0) = 0)
        \\ =& \frac{x}{r} + \left( 1 - \frac{x}{r} \right) u(0)
    \end{align*}
    and, for $j \not= e$,
    \begin{align*}
        u_j(x) = \bP( T_0 < T_r \mid B(0) = x) \bP( T^\star = T_{(e,x)} \mid \cZ(0) = 0)
        = \left( 1 - \frac{x}{r} \right) u(0)
    \end{align*}
    which implies that $u_j(x) \to u(0)$ as $x \to 0$, for every $j$.
    \\
    Notice that by symmetry, the probability that the process starting in the origin reaches level $r$ along the edge $e$ equals to $p_e$:
    \begin{align*}
    u(0) = p_e.
    \end{align*}
    \\
    In particular, from the boundary conditions, we get
    \begin{align*}
        u_j(r) = 0, \quad u_j(0) = p_e \quad &\Longrightarrow \quad u_j(x) = p_e \left(1 - \frac{x}{r} \right), \qquad j \not= e,
        \\
        u_e(r) = 1, \quad u_e(0) = p_e \quad &\Longrightarrow \quad u_e(x) = p_e + \left(1 - \frac{x}{r} \right) (1 - p_e)
    \end{align*}
    We finally obtain, from previous representation, that
    \begin{align*}
    \frac12 \frac{c}{b} u''(0) = 0 = \sum_{j \not= e} p_j \left( -\frac{p_e}{r} \right) + p_e \frac{1-p_e}{r}.
    \end{align*}
    Hence $u \in D(A_\cX)$ and the proof is complete.
\end{proof}

\begin{remark}
The general solution of problem (D) is given by a linear combination of the $n$ different solutions that we obtain from previous theorem by rotating the values of $e$.
\end{remark}

\subsubsection{Poisson problem}

It is customary to identify functions that satisfy problem \eqref{e.D} with the {\em harmonic} functions on $\mE$.
In this section, we consider the realted {\em Poisson} problem
associated with $A_\cX$
\begin{align*}
\tag{P}\label{e.P}
\begin{cases}
\text{find $v \in D(A_\cX)$ such that}\\
\frac12 A_\cX v(\mx) = -1, \qquad \mx \in B_r \\
u_j(r) = 0, \qquad j = 1, \dots, n.
\end{cases}
\end{align*}

In this section we prove that the solution of the Poisson problem is related to the mean exit time from the ball $B_r(0)$
\begin{align}
\label{e.def.v}
    v(\mx) = \bE [ T^\star \mid \cZ(0) = \mx].
\end{align}

\begin{theorem}\label{th.P}
    The function $v(\mx)$ defined in \eqref{e.def.v} satisfies the Poisson problem (P) on the ball $B_r(0)$.
\end{theorem}

\begin{proof}
    At first, we notice that $v(0)$ is equal to the first passage time from level $r > 0$ for a sticky Brownian motion $X(t)$ in $[0, \infty)$ starting from $0$, since we are only interested in the radial part of the process $\cX(t)$.
    It holds, by \cite{BRHC2020}
    \begin{align*}
        v(0) = \bE[T^\star \mid \cZ(0) = 0] = \bE[T_r \mid X(0) = 0] = r^2 +  \frac{c}{b} r.
    \end{align*}
    Next, let $\mx = (j, x)$ and choose $h$ small enough that $(x-h, x+h) \subset (0,r)$. Then
    \begin{align*}
    v_j(x) =& 
    \bE[ T^\star - T_{(0,r)} \mid \cX(0) = \mx ] + \bE[ T_{(0,r)} \mid \cX(0) = \mx ]
    \\ =&
    \bE[ (T^\star - T_{(0,r)}) \uno_{\{T_{(0,r)} = T_0\}} \mid \cX(0) = \mx ] + \bE[ T_{(0,r)} \mid \cX(0) = \mx ]
    \\ =&
    \bE[ T^\star - T_{(0,r)} \mid \cX(T_{(0,r)} = 0, \cX(0) = \mx ] \bP( X(T_{(0,r)} = 0 \mid X(0) = \mx ) + \bE[ T_{(0,r)} \mid \cX(0) = \mx ]
    \\
    \intertext{In the first expectation, we use the strong Markov property of the sticky Brownian motion; in the second and last term, we use the known results for a one-dimensional Brownian motion to get}
    v_j(x) =& v(0) \frac{r-x}{r} + x(r-x).
    \end{align*}
   We observe, passing by, that the function $v_j(x)$ is independent of $j$, i.e., the solution is homogeneous in the various edges.
   Since $v(0)$ is known, we have an explicit representation for the solution
   \begin{align*}
   v_j(x) = r^2 - x^2 + \frac{ c}{b} (r - x),
   \end{align*}
   which implies $v_j(0) = v(0)$ and $v''_j(0) = -2$ for every $j$, hence $v \in C^2_0(\mE)$ and  $\frac12 v_j''(x) = -1$, hence the equation in problem \eqref{e.P} is satisfied; it remains to check, in order to prove that $v \in D(A_\cX$), the condition in 0
   \begin{align*}
   \frac12 c v''(0) = -c = b \sum_{k=1}^n p_k \left( - \frac{ c}{b}  \right)
   \end{align*}
   and the proof is complete.
\end{proof}

\begin{remark}
A similar result holds for the Brownian motion $\cZ(t)$, which means taking $c=0$ in the definition of the domain, and in which case it holds $\bE[ T^\star \mid \cZ(0) = \mx] = r^2 - |x|^2$.
\end{remark}

\section{Sticky Brownian motion with trapping star vertex}
\label{sez4}

In this section, we modify the time change used in \eqref{eq:timechange-SBM} to construct the sticky Brownian motion, in order to allow a \emph{random} time change of the form
\begin{align}
\label{eq:timechange-SBM-2}
V_H(t) = t + H \circ \mu \,  \ell^\cZ(t), \qquad t \ge 0,
\end{align}
where $\ell^\cZ$ is the local time at $0 \in \mE$ of the standard Brownian motion on $\cZ$ on $\mE$ and $H = \{H_t,\ t \ge 0\}$ is a subordinator independent from $\cZ$. We denote by
\begin{align*}
\Phi(\lambda) = \int_0^\infty \left( 1 - e^{-\lambda y} \right) \phi({\rm d}y), \quad \lambda>0
\end{align*}
the symbol of $H$ written in terms of the so-called associated L\'{e}vy measure $\phi$. Thus, it holds that
\begin{align*}
\mathbb{E}[e^{ -\lambda H_t} ] = e^{-t \Phi(\lambda)}, \quad \lambda>0,\; t\geq 0.
\end{align*}
The inverse process $L$ is defined as $L_t \colon= \inf\{s \ge 0\ :\ H_s > t\}$. 
The process $H$ has strictly increasing path and continuous right-inverse $L$. Moreover, the subordinator $H$ may have jumps, so that the inverse $L$ may have plateaux.
We also remark the known result that the stable subordinator is identified with the symbol $\Phi(\lambda) = \lambda^\alpha$, for $0 < \alpha < 1$.

The resulting process $\cY(t) = \cZ(V_H^{-1}(t))$ can be associated with a sticky Brownian motion and can be regarded as a Brownian motion with trap in the origin $0 \in \mE$. In particular, we are interested in the (local) Cauchy problem with non-local condition at the origin that we may associate to this process.

Recall that a standard Brownian motion $\cZ$ on the star graph has a radial component $Z$ that is a reflected Brownian motion on the positive half-line.
We shall use consistent notation for the two objects (so we denote the local times $\ell^\cZ$ and $\ell^Z$, first passage times from the vertex/origin $\cT_0$ and $T_0$).

In the one dimensional case, we know that the inverse to $V(t) = t + \mu \ell^{Z}(t)$ slows down the reflecting Brownian motion $Z$ at the origin. 
Thus, $Y(t) = Z \circ V_H^{-1}(t)$ is forced to stop for a random amount of time at the origin. 
Since $H$ is independent from the couple $(Z, {\ell^{Z}})$,  the holding time at the sticky point 0 is independent from $Z$. 

\begin{remark}
\label{remark:holdingT}
{\color{black}
If we assume that $\tau$ is the holding time (at zero) for the sticky Brownian motion $X=\{X_t, t\geq0\}$ on $[0, \infty)$, then 
\begin{align*}
\mathbb{P}_0(\tau>t | X_{\tau}) = \exp(- \mu t)
\end{align*}
for the positive rate $\mu = \frac{b}{c}$. This result has been discussed in \cite{Ito1963}: the exponential law guarantees the semigroup property and the definition of holding time follows by considering that, with the process $X$ starting at $x=0$, the probability $\mathbb{P}_0(\tau>t | X_\tau >0)$ describes the time the process spends at $x=0$. 
Moreover, the process  enjoys the Markov property and the sequence $\{\tau^i, \, i \in \mathbb{N}\}$ of holding times for $X$ are independent and identically distributed. 
For the process $Y$ we can introduce the sequence $\{\tau_Y^i, \; i \in \mathbb{N}\}$ of holding times for which (see \cite{FBVP2})
\begin{align*}
\mathbb{P}_0(\tau_Y^i > t | Y_{\tau_Y^i} >0) = \mathbb{P}_0(H_{\tau^i} > t | Y_{\tau_Y^i} >0) = \mathbb{P}_0(\tau^i > L_t | Y_{\tau_Y^i} >0)
\end{align*}
where we used the fact that $L=H^{-1}$ is and inverse process. As $\Phi(\lambda) =\lambda$ the process $H_t$ becomes the elementary subordinator and $Y_t = X_t$ in law. 
Since $H$ is independent from $Z$, then $H$ is independent from $\tau$. 
In particular, the process $Y$ moves on the path of $X$ (or $Z$) but it stops at $x=0$ for a longer amount of time according with the new holding time $\tau_Y^i = H \circ \tau^i$. 
We can therefore write  
\begin{align}
\label{holdingY}
\mathbb{P}_0(\tau^i > L_t | Y_{\tau_Y^i} >0) = \mathbb{P}_0(\tau^i > L_t | X_{\tau^i} >0) = \mathbb{E}_0[\exp(- \mu L_t)], \quad \forall\, i.
\end{align}
The holding times $\tau_Y^i$ are independent and identically distributed:
the independence follows immediately by observing that
\begin{align*}
H_{\tau^1} = H_{\tau^0 + \tau^1} - H_{\tau^0} \perp H_{\tau^1 + \tau^2} - H_{\tau^1} = H_{\tau^2}
\end{align*}  
where we used the properties of the subordinator $H$. 
We refer to \cite[Lemma 6]{FBVP2} for a detailed discussion. 
The process $L$ depends on the symbol $\Phi$ and we can study the mean amount of time the process spends on the sticky point in terms of $\Phi$, that is we may have finite and infinite mean amount of time (at $x=0$) and in case of infinite holding time we are able to characterize the tail behavior in \eqref{holdingY}. 

We conclude our discussion by recalling that, in case $\Phi(\lambda)=\lambda^\alpha$, that is $H$ is a stable subordinator, we have that
\begin{align*}
\mathbb{E}_0[\exp(- (b/c) L_t )] = E_\alpha(- (b/c) t^\alpha) = \sum_{k \geq 0} \frac{(-(b/c) t^\alpha)^k}{\Gamma(\alpha k +1)}
\end{align*}
is the well-known Mittag-Leffler function. We know that $E_\alpha \notin L^1(0,\infty)$ and the mean amount of (holding) time given by $\mathbb{E}_0[\tau^i_Y]$ is infinite. The process $Y$ spends an infinite mean amount of time at $x=0$.

In general the behavior on the boundary point (or boundary set, in higher dimension) can be associated with a delayed or a rushed effect depending on $\Phi$ as discussed in \cite{CapDovALEA}.
}
\end{remark}

\subsection{Probabilistic construction}

In the next result we discuss the equivalence of the excursions between the processes $\cY$ and $Y$ during an excursion (i.e., far from the origin).

\begin{lemma}
Assume that $\cY(0) = (e, x)$ with $x > 0$.
Then, for every function $f \in C_0(\mE)$ and $\lambda > 0$ it holds
\begin{align*}
\bE^{(e,x)} \int_0^{\cT_0} e^{-\lambda t} f(\cY_t) \, {\rm d}t = \bE^{x} \int_0^{T_0} e^{-\lambda t} f(e, Y_t) \, {\rm d}t.
\end{align*}
\end{lemma}

\begin{proof}
Let us briefly remark that if $f \in C_0(\mE)$ then $f(e, \cdot) \in C_0(\bR_+)$ for every $e \in \{1, \dots, n\}$.
We start by a change of variable
\begin{align*}
\bE^{(e,x)} \int_0^{\cT_0} e^{-\lambda t} f(\cY_t) \, {\rm d}t 
= \bE^{(e,x)} \int_0^{\cT_0} e^{-\lambda t} f(\cZ \circ V^{-1}_H(t)) \, {\rm d}t 
= \bE^{(e,x)} \int_0^{V_H(\cT_0)} e^{-\lambda V_H(t)} f(\cZ_t) \, {\rm d}V_H(t) 
\end{align*}
but on $[0, \cT_0)$ we have $\ell^\cZ(t) = 0$, hence $V_H(t) = t$ and we get
\begin{align*}
\bE^{(e,x)} \int_0^{\cT_0} e^{-\lambda t} f(\cY_t) \, {\rm d}t = \bE^{(e,x)} \int_0^{\cT_0} e^{-\lambda t} f(\cZ_t) \, {\rm d}t.
\end{align*}
Now, since we have (by construction!) equivalence between excursions of the processes $\cZ$ and $Z$, we have
\begin{align*}
\bE^{(e,x)} \int_0^{\cT_0} e^{-\lambda t} f(\cZ_t) \, {\rm d}t = \bE^{x} \int_0^{T_0} e^{-\lambda t} f(Z_t) \, {\rm d}t,
\end{align*}
and the thesis follows by using analog equalities for the reflected Brownian motion.
\end{proof}

\begin{theorem}
\label{theorem:YXbar}
The $\lambda$-potential of the process $\cY$ equals the resolvent operator of the process $\tilde \cX$, that is a sticky Brownian motion on the star graph (according to Definition \ref{de:SBM})
with parameters $b'$ and $c'$ such that
\begin{align*}
\frac{c'}{b'} = \frac{c}{b} \frac{\Phi(\lambda)}{\lambda}.
\end{align*}
\end{theorem}

\begin{proof}
The proof is based on a direct computation, and makes use of the independence between the subordinator and the Brownian motion $\cZ$. Assume that $f$ is a continuous and bounded function on $\mE$. We start by computing
\begin{align*}
\cU_\lambda f(\mx) &= \bE^{\mx}\left[ \int_0^\infty e^{-\lambda t} f(\cY_t) \, {\rm d}t \right]
\\
&= \bE^{\mx} \left[ \int_0^{T_0} e^{-\lambda t} f(\cZ \circ V^{-1}_H(t)) \, {\rm d}t \right] + \bE^{\mx} \left[ e^{-\lambda T_0} \right] \, \bE^{0} \left[ \int_0^t e^{-\lambda t} f(\cZ \circ V^{-1}_H(t)) \, {\rm d}t \right]
= \cU_\lambda^1 f(\mx) + \cU^2_\lambda(\mx).
\end{align*}
Recall that $V_H(t) = t + H \circ \mu \ell^\cZ(t)$, hence on $[0,T_0)$ it holds $V_H(t) = t$; therefore
\begin{align*}
\cU_\lambda^1 f(\mx) = \bE^{\mx} \left[ \int_0^{T_0} e^{-\lambda t} f(\cZ(t)) \, {\rm d}t \right].
\end{align*}
Moreover, a standard computation leads to 
\begin{align*}
\bE^{\mx} \left[ e^{-\lambda T_0} \right] = e^{-\sqrt{2 \lambda} x}.
\end{align*}
It remains to examine the last term. Since $V_H(t)$ is a continuous and strictly increasing, the same holds for its inverse process, and we can write
\begin{multline*}
\bE^{0} \left[ \int_0^\infty e^{-\lambda t} f(\cZ \circ V^{-1}_H(t)) \, {\rm d}t \right]
= \bE^{0} \left[ \int_0^\infty e^{-\lambda V_H(t)} f(\cZ(t)) \, {\rm d}V_H(t) \right]
= \bE^{0} \left[ -\frac{1}{\lambda} \int_0^\infty f(\cZ(t)) \, {\rm d}e^{-\lambda V_H(t)} \right]
\\
= -\frac{1}{\lambda} \bE^0 \left[ \left.e^{-\lambda V_H(t)} f(\cZ(t))\right|_{t=0}^\infty - \int_0^t e^{-\lambda V_H(t)} \, {\rm d} f(\cZ(t)) \right] = \frac{1}{\lambda} \bE^0 \left[  f(0) + \int_0^t e^{-\lambda V_H(t)} \, {\rm d} f(\cZ(t)) \right]
\end{multline*}
Now, since $H$ and $\cZ$ are independent, by taking a conditional expectation in the last term it is possible to write
\begin{align*}
\bE^0 \left[  \int_0^t e^{-\lambda V_H(t)} \, {\rm d} f(\cZ(t)) \right] = \bE^0 \left[  \int_0^t \bE \left[ e^{-\lambda V_H(t)} \right] \, {\rm d} f(\cZ(t)) \right] = \bE^0 \left[  \int_0^t  e^{-\lambda t -\Phi(\lambda) \mu \ell^{\cZ}_t}  \, {\rm d} f(\cZ(t)) \right].
\end{align*}
Let us define a time change
\begin{align*}
T(t) = t + \mu \frac{\Phi(\lambda)}{\lambda} \ell^\cZ_t;
\end{align*}
a second application of the integration by parts formula implies
\begin{multline*}
\bE^{0} \left[ \int_0^\infty e^{-\lambda t} f(\cZ \circ V^{-1}_H(t)) \, {\rm d}t \right]
= \frac{1}{\lambda} \bE^0 \left[  f(0) + \int_0^t e^{-\lambda T(t)} \, {\rm d} f(\cZ(t)) \right]
\\
= \frac{1}{\lambda} \bE^0 \left[   \int_0^t e^{-\lambda T(t)} f(\cZ(t)) \, {\rm d}T(t) \right] = \frac{1}{\lambda} \bE^0 \left[   \int_0^t e^{-\lambda t} f(\cZ \circ T^{-1}(t)) \, {\rm d}t \right].
\end{multline*}
Summing up, we obtain
\begin{multline*}
\cU_\lambda f(\mx) = \bE^{\mx} \left[ \int_0^{T_0} e^{-\lambda t} f(\cZ \circ T^{-1}(t)) \, {\rm d}t \right] + \bE^{\mx} \left[ e^{-\lambda T_0} \right] \, \bE^{0} \left[ \int_0^t e^{-\lambda t} f(\cZ \circ T^{-1}(t)) \, {\rm d}t \right]
\\
= \bE^{\mx}\left[ \int_0^\infty e^{-\lambda t} f(\tilde \cX_t) \, {\rm d}t \right].
\end{multline*}
Notice that the relations
\begin{align*}
\frac{c'}{b'} = \frac{c}{b} \frac{\Phi(\lambda)}{\lambda}, \qquad b' + c' = 1
\end{align*}
univocally identify the sticky Brownian motion $\tilde \cX$.
\end{proof}


{\subsection{Non-local operators in time with dynamic conditions}}

In the literature, several alternative definitions and formulations of fractional derivatives have been proposed,
such as the Riemann-Liouville \cite{Podlubny1999} and Gr\"unwald-Letnikov \cite{Diethelm2004} derivatives;
in this paper, we consider a Caputo-Djrbashian type operator associated with the L\'evy measure $\phi$ of a subordinator $H$ through the formula
\begin{align}
\label{eq:defCDop}
\mathfrak{D}^\Phi_t u(t, x) = \int_0^t \frac{\partial u}{\partial s}(s,x)\, \overline{\phi}(t-s)\,ds, \quad t>0, \; x \in D
\end{align}
where $\overline{\phi}(z)=\phi(z,\infty)$ is the tail of $\phi$. 
The operator $\mathfrak{D}^\Phi_t$ coincides with the well-known Caputo or Caputo-Djrbashian derivative as $\Phi(z)=z^\alpha$ with $\alpha\in (0,1)$ which is the case of stable subordinators. The convolution-type operator known as Caputo-Djrbashian derivative has been introduced by the first author in the works \cite{caputoBook, CapMai71,CapMai71b} and by the second author who actively investigated this operator starting from the papers \cite{Dzh66, DzhNers68}. The general operator in \eqref{eq:defCDop} has been considered in \cite{Koc2011} and after in \cite{Chen17, Kolo19, Toaldo2015}.

It is well-known that the relation between the fractional derivative operator $\mathfrak{D}^\Phi_t$ and the subordinator $H$ (and its inverse $L$)
allows an analysis of  PDEs with local boundary conditions and the probabilistic representation of their solutions. The well-known theory can be referred to as non-local initial value problems or non-local Cauchy problems. Here we deal with local problems equipped with non-local boundary conditions, we say non-local boundary value problems. Despite the vast contributes on non-local initial value problems, the literature on non-local boundary value problems seems to be very lacking.

Non-local initial value problems \emph{on the positive half line} involving such a operator have been considered for example in \cite{Chen2017, Koc2011, Toaldo2015}. Their definitions of $\mathfrak{D}^\Phi_t$ slightly differ as well as the characterization of their results. 

A standard condition for \eqref{eq:defCDop} to be well defined is usually given by requiring that $t \mapsto u(t, \cdot)$ belongs to the set $W^{1,\infty}(0, \infty)$ of essentially bounded functions with essentially bounded derivatives. This requirement well agrees with the Laplace machinery. Indeed, by considering that \eqref{eq:defCDop} is defined as a convolution-type operator, we get
\begin{align*}
\int_0^\infty e^{-\lambda t} \mathfrak{D}^\Phi_t u(t, x)\, dt 
& = \big(\lambda u(\lambda, x) - u(0,x) \big) \left( \int_0^\infty e^{-\lambda t} \bar{\phi}(t)dt \right)
\end{align*}
where (\cite{Bertoin1999})
\begin{align}
\int_0^\infty e^{-\lambda t} \bar{\phi}(t)dt = \frac{\Phi(\lambda)}{\lambda}, \quad \lambda>0
\label{tailLap}
\end{align} 
and $u(\lambda, x)$ is the Laplace transform of $u(t,x)$. If $u, u^\prime$ are bounded, then the Laplace transforms of $u, u^\prime$ are well-defined. Thus, we consider $u \in W^{1, \infty}(0, \infty) \cap C(D)$ for a bounded set $D \subset \mathbb{R}^d$, $d\geq 1$. We introduce a further characterization by asking for the following condition to be satisfied:
\begin{align}
\label{condMD}
\exists\, M_D>0\,:\, \bigg| \frac{\partial u}{\partial s}(s,x) \bigg| \leq  M_D\, \frac{\kappa(ds)}{ds}
\end{align} 
where 
\begin{align*}
\kappa(ds) = \int_0^\infty \mathbb{P}^0(H_t \in ds) dt
\end{align*}
is the potential measure for the subordinator $H$ with symbol $\Phi$. Since $\kappa$ and $\bar{\phi}$ are associated Sonine kernels for which 
\begin{align*}
\int_0^t \bar{\phi}(t-s) \kappa(ds) =1
\end{align*} 
and
\begin{align}
\label{unifBoundD}
| \mathfrak{D}^\Phi_t u(t, x)| \leq M_D \int_0^t \bar{\phi}(t-s) \kappa(ds),
\end{align} 
then we obtain that $|\mathfrak{D}^\Phi_t u(t,x)|$ is uniformly bounded on $(0, \infty) \times D$.

Moving on the star graph, for the operator
\begin{align*}
\mathfrak{D}^\Phi_t u(t, \mx) = \int_0^t \frac{\partial u}{\partial s}(s, \mx) \bar{\phi}(t-s)ds, \quad t>0,\; \mx \in \mE
\end{align*}
we may consider $u \in W^{1, \infty}(0, \infty) \cap C_b(\mE)$. By following the previous arguments, we consider the following condition:
\begin{align}
\label{condME}
\exists\, M_{\mE}>0\,:\, \bigg| \frac{\partial u}{\partial s}(s,\mx) \bigg| \leq  M_{\mE}\, \frac{\kappa(ds)}{ds}.
\end{align} 

\begin{remark}
For the positive solutions $u(t,x)$ and $u(t,\mx)$ respectively under \eqref{condMD} and \eqref{condME}, we observe that:
\begin{itemize}
\item[i)] $u(s,x) \leq M_{D} \, \kappa((0,s]) = M_{D} \, \mathbb{E}^0[L_s], \; s\geq 0,\; x \in D$;
\item[ii)] $u(s,\mx) \leq M_{\mE}\,  \kappa((0,s]) = M_{\mE} \, \mathbb{E}^0[L_s], \; s\geq 0,\; \mx \in \mE$.
\end{itemize}
Indeed,
\begin{align*}
\kappa((0,s]) = \int_0^s \kappa(dz) = \int_0^\infty \mathbb{P}^0(H_t < s) dt = \int_0^\infty \mathbb{P}^0(t < L_s) dt = \mathbb{E}^0[L_s] 
\end{align*}
where we used the fact that $L$ is the inverse of $H$.
\end{remark}

We are now ready to focus on the problem
\begin{align}
\tag{NL}\label{problemNLBVP}
\left\lbrace
\begin{array}{ll}
\displaystyle \frac{\partial u}{\partial t}(t, \mx) = \frac12 \frac{\partial^2 u}{\partial x^2}(t, \mx), \quad t>0, \; \mx \in \mE \setminus \{0\}\\
\\
\displaystyle c \, \mathfrak{D}^\Phi_t u(t,0) = b \sum_{k=1}^n p_k u_k'(t,0) , \quad t>0 \\
\\
\displaystyle u(0, \mx) = u_0(\mx), \quad x \in \mE, \quad u_0 \in C_0(\mE)
\end{array}
\right .
\end{align}
which involves a non-local operator in the boundary condition as a non-local dynamic condition. Thus, we are dealing with a non-local boundary value problem. 

A diffusion problem on the half-line, with fractional dynamic boundary condition, described in terms of a time fractional derivative (the Caputo derivative $D^\alpha_t$ depending on $\Phi(z)=z^\alpha$, $\alpha \in (0,1)$) has been recently introduced in  \cite{FBVP1, FBVP2} and further extended to non-local diffusion problem in \cite{Colantoni2022}. We extend the construction provided in these papers, as our aim is to show that the solution to the non-local boundary value problem \eqref{problemNLBVP} can be written as
\begin{align}
\label{solNLBVP}
u(t,\mx) = \mathbb{E}^{\mx} [ u_0(\cY_t)] = \mathbb{E}^{\mx}\left[ u_0(\cZ \circ V_H^{-1}(t)) \right], \quad t>0,\; \mx \in \mE
\end{align}
where $\cZ$, $V_H$ and $\cY = \cZ \circ V_H^{-1}$ have been previously introduced. We also recall the $\lambda$-potential
\begin{align*}
\cU_\lambda u_0(\mx) &= \bE^{\mx}\left[ \int_0^\infty e^{-\lambda t} u_0(\cY_t) \, {\rm d}t \right], \quad \lambda>0, \; \mx \in \mE.
\end{align*}

Let us consider the space
\begin{align*}
D_L := \bigg\{ \varphi \in C((0, \infty) \times \mE) \textrm{ with } \varrho = \varphi |_{\mx =0} \textrm{ such that } \varrho, \frac{d \varrho}{d t}, \mathfrak{D}^\Phi_t \varrho \in C(0, \infty) \textrm{ and }   \bigg| \frac{d \varrho}{d t}(t) \bigg| \leq M_0 \frac{\kappa(dt)}{dt}, \; M_0>0\bigg\}.
\end{align*}

\begin{theorem}
\label{thm:NLBVP}
The solution $u \in C^{1,2}((0, \infty) \times \mE) \cap D_L$ to the problem \eqref{problemNLBVP} has the probabilistic representation \eqref{solNLBVP}.
\end{theorem}

\begin{proof}
First we write $\cU_\lambda u_0 = \cU^1_\lambda u_0 + \cU^2_\lambda u_0$ as in the proof of Theorem \eqref{theorem:YXbar}. Now notice that $\cU^1_\lambda u_0$ belongs to the domain of the Dirichlet Laplacian (with Dirichlet boundary condition at $0=(j,0)$ $\forall\, j \in \{1,2,\ldots, n\}$). That is
\begin{align*}
A \cU^1_\lambda u_0(\mx) = \lambda \cU^1_\lambda u_0(\mx) - u_0(\mx)
\end{align*}
(where, as usual, $A$ is the differential operator $A f(\mx) = \frac12 f_j''(x)$, $j=1,2, \ldots, n$) and $\cU^1_\lambda$ is the resolvent operator of the Brownian motion killed in the origin. A direct computation leads to
\begin{align*}
A \cU^2_\lambda u_0(\mx) = \lambda \cU^2_\lambda u_0(\mx);
\end{align*}
therefore
\begin{align*}
A \cU_\lambda u_0(\mx) = \lambda \cU_\lambda u_0(\mx) - u_0(\mx)
\end{align*}
and we identify the heat equation on $\mE \setminus \{0\}$.

Next, we determine the boundary condition. In particular, by Definition \ref{de:SBM} and Theorem \ref{theorem:YXbar},
\begin{align}
\label{domainPotentialU}
\cU_\lambda u_0 \in  \left\{f \in C_0^2(\mE) \ :\ \frac12 c \frac{\Phi(\lambda)}{\lambda}  f''(0) = b \sum_{k=1}^n p_k f'_k(0) \right\}.
\end{align}
Indeed, we recall from formula \eqref{eq:ident3} applied to the sticky Brownian motion $\tilde \cX$ that
\begin{align*}
\left( \lambda + \frac{b'}{c'} \sqrt{2 \lambda} \right) \cU_\lambda f(0) = f(0) + \frac{2 b'}{c'} \sum_{k=1}^n p_k \hat f_k(\sqrt{2 \lambda})
\end{align*}
which, by passing through 
\begin{align*}
\lambda\cU_\lambda f(0) - f(0) = \frac{b'}{c'} \left( - \sqrt{2 \lambda}  \cU_\lambda f(0) + \sum_{k=1}^n p_k \hat f_k(\sqrt{2 \lambda}) \right),
\end{align*}
takes the form
\begin{align*}
\frac{\Phi(\lambda)}{\lambda} \left(\lambda\cU_\lambda f(0) - f(0) \right) = \frac{b}{c} \left( - \sqrt{2 \lambda}  \cU_\lambda f(0) + \sum_{k=1}^n p_k \hat f_k(\sqrt{2 \lambda}) \right).
\end{align*}
By comparison, we recognize in the left hand side the Laplace transform of the non-local operator $\mathfrak{D}^\Phi_t u(t, 0)$ and in the right hand side (see the computation in Lemma \ref{lemma10}) that of 
\begin{align*}
\frac{b}{c} \sum_{k=1}^n p_k u_k'(t,0),
\end{align*}
which implies \eqref{domainPotentialU}. As simple arguments show, we observe that, $\forall\, t>0$,
\begin{align*}
\lim_{\mx \to 0\in \mE} \mathfrak{D}^\Phi_t u(t, \mx) = \mathfrak{D}^\Phi_t \varpi(t)
\end{align*}
where
\begin{align*}
\varpi(t) =  \lim_{\mx \to 0\in \mE} u(t, \mx).
\end{align*}
Since $u \in D_L$, then $\mathfrak{D}^\Phi_t \varpi$ is well-defined.
This identifies the non-local (dynamic) equation on the vertex $0\in \mE$. 

Notice that $u_0 \in C_0(\mE)$ implies $\cU^1_\lambda u_0 \in C^2_0(\mE)$ via Dirichlet semigroup. Moreover, $\cU_\lambda u_0 \in C^2_0(\mE)$. This is a direct consequence of the equivalence between $\cX$ on $\mE$ and $X$ on $[0, \infty)$. Indeed, for the generator $A_X$ of $X$ we have $D(A_X) \subset C^2_0([0, \infty))$. 

Uniqueness follows from the Laplace techniques: there exists at most one continuous inverse, since our inverse $u$ to $\cU_\lambda u_0$ is continuous, then $u$ is unique.

\end{proof}


\begin{theorem}
\label{thm:HoldingTime}
For the sequence of holding times $\{\tau^i\}_i$ at $0 \in \mE$ of the process $\cY$ on $\mE$, it holds that:
\begin{itemize}
\item[i)] $\tau^i$ are i.i.d. random variables whose distribution is  
given below;
\item[ii)] $\mx = 0 \in \mE$ implies
\begin{align*}
\mathbb{P}^{\mx}(\tau^1 >t \,|\, \cY_{\tau^1} \in \mE \setminus \{\mx\}) = \mathbb{E}^0 [\exp (- \mu L_t) ], \quad t>0;
\end{align*}
\item[iii)] $\mathbb{E}[\tau^1] <\infty$ iff $\mu = \frac{b}{c} < \infty$ and
\begin{align*}
\lim_{\lambda \to 0} \frac{\Phi(\lambda)}{\lambda} < \infty;
\end{align*}
\item[iv)] $\mathbb{E}[\tau^1]<\infty$ and $t \mapsto u(t, \cdot)$ in $W^{1,1}(0, \infty)$ imply that $t \mapsto \mathfrak{D}^\Phi_t u(t, \cdot)$ is in $L^1(0, \infty)$;
\item[v)] $\mathbb{E}[\tau^1] >0$ and $t \mapsto \dot{u}(t, \cdot)$ is bounded ($u \in D_L$) imply that $t \mapsto \mathfrak{D}^\Phi_t u(t, \cdot)$ is bounded (uniformly bounded).
\end{itemize}
\end{theorem}

\begin{proof}
The holding time for $\cY$ on the vertex $0 \in \mE$ is given by the holding time of $Y$ at zero. For the process $Y$ we have introduced the sequence $\{\tau_Y^i, \; i \in \mathbb{N}\}$ of holding times for which (see Remark \ref{remark:holdingT})
\begin{align*}
\mathbb{P}^0(\tau_Y^i > t | Y_{\tau_Y^i} >0) = \mathbb{P}^0(\tau^i > L_t | Y_{\tau_Y^i} >0).
\end{align*}
Since $Y$ moves along the path of $X$, we have the equivalence $(Y_{\tau_Y^i} >0) \equiv (X_{\tau^i_X} >0)$ where, here, $\{\tau^i_X\}_i$ is the sequence of holding times for $X$ introduced in Remark \eqref{remark:holdingT}. Thus, we get formula \eqref{holdingY},
\begin{align}
\mathbb{P}^0(\tau^i > L_t | Y_{\tau_Y^i} >0) = \mathbb{E}^0[\exp(-\mu L_t)], \quad \forall\, i.
\end{align}
On each edge $\mathbf{e}_j \in \mathcal{E}$, we can therefore write
\begin{align*}
\mathbb{P}^0(\tau^i > t\, |\, \cY_{\tau^i} \in \mathbf{e}_j \setminus \{0\}) = \mathbb{P}^0(\tau_Y^i > t | Y_{\tau_Y^i} >0) = \mathbb{E}^0[\exp(-\mu L_t)], \quad j=1,2, \ldots, n, \quad \forall\, i.
\end{align*}
In particular, 
\begin{align*}
\mathbb{P}^0(\tau^i > t\, |\, \cY_{\tau^i} \in \mE \setminus \{0\}) = \sum_{j=1}^n p_j\, \mathbb{E}^0[\exp(- \mu L_t)] = \mathbb{E}^0[\exp(- \mu L_t)], \quad \forall\, i 
\end{align*}
and we get the claim.

The point {\it iii)}  can be proved by observing that
\begin{align*}
\int_0^\infty e^{-\lambda t} \, \mathbb{E}^0[\exp (- \mu L_t)] \, dt = \frac{\Phi(\lambda)}{\lambda} \frac{1}{\mu + \Phi(\lambda)}, \quad \lambda>0.
\end{align*}
As $\lambda \to 0$ we get $\mathbb{E}[\tau]$. Since $\Phi(0)=0$, we only need to check for the limit of $\Phi(\lambda)/\lambda$ as $\lambda \to 0$. 

For the point {\it iv)} we first notice that
\begin{align*}
\int_0^\infty e^{-\lambda t} \, \mathfrak{D}^\Phi_tu(t, \mx)  \, dt = \left( \int_0^\infty e^{-\lambda t} \, \frac{\partial u}{\partial t}(t, \mx) \, dt \right) \left( \int_0^\infty e^{-\lambda t} \, \bar{\phi}(t) \, dt \right)
\end{align*}
and
\begin{align*}
|\mathfrak{D}^\Phi_tu(t, \mx) | \leq \int_0^t \bigg| \frac{\partial u}{\partial s}(s, \mx) \bigg| \, \bar{\phi}(s) \, ds, \quad t>0,\, \mx \in \mE.
\end{align*}
Thus, we get, for $\mx \in \mE$,
\begin{align*}
\int_0^\infty | \mathfrak{D}^\Phi_tu(t, \mx) | dt \leq  \left( \int_0^\infty \bigg| \frac{\partial u}{\partial t} (t,\mx) \bigg| dt \right) \left( \lim_{\lambda \to 0} \int_0^\infty e^{-\lambda t} \bar{\phi}(t)dt \right) = \bigg\|\frac{\partial u}{\partial t}(\cdot, \mx) \bigg\|_{L^1(0, \infty)} \left( \lim_{\lambda \to 0} \frac{\Phi(\lambda)}{\lambda} \right). 
\end{align*}
Since $u$ solves the heat equation on $\mE \setminus \{0\}$ and $Au \in C(\mE)$ we write
\begin{align*}
\| \mathfrak{D}^\Phi_tu(\cdot, 0) \|_{L^1(0, \infty)} \leq \bigg\| \frac{\partial u}{\partial t}(\cdot , 0) \bigg\|_{L^1(0, \infty)} \left( \lim_{\lambda \to 0} \frac{\Phi(\lambda)}{\lambda} \right).
\end{align*}
Assume that $\Phi(\lambda) /\lambda$ is finite as $\lambda \to 0$. We conclude that $u(\cdot, 0) \in W^{1,1}(0, \infty)$ implies $\mathfrak{D}^\Phi_t u(\cdot , 0) \in L^1(0, \infty)$. 

Point {\it v)} basically says that we have no restriction on the symbol $\Phi$. Indeed, $\forall\, \Phi$, that is for $\mathbb{E}[\tau]>0$, Theorem \ref{thm:NLBVP} and formula \eqref{unifBoundD} hold true. This is the case $u \in D_L$. In case $t \mapsto \dot{u}(t, \cdot)$ is bounded (for example of exponential order $w>0$, $|\dot{u}| \leq M e^{wt}$ ) we simply get, at $\mx=0$ for instance, 
\begin{align*}
|\mathfrak{D}^\Phi_t u(t, 0)| \leq M \int_0^t e^{ws} \bar{\phi}(t-s)ds < \infty.
\end{align*}
\end{proof}

%

\begin{remark}
Let us consider $\Phi(\lambda) = a\ln (1+\lambda/b)$. We remark that
\begin{align*}
\lim_{\lambda \to 0} \frac{\Phi(\lambda)}{\lambda} = \frac{a}{b}<\infty.
\end{align*}
Thus, the mean holding time is finite in case of Gamma subordinators with $a,b \in (0, \infty)$. 
\end{remark}

\begin{remark}
Observe that $t \mapsto u(t, \cdot) \in W^{1,1}(0, \infty)$ implies $t \mapsto u(t, \cdot) \in L^\infty(0, \infty)$. Thus, in point {\it iv)} of Theorem \ref{thm:HoldingTime} we are still working with bounded functions.
\end{remark}


\section*{Acknowledgments}
The first author would like to thank the group INdAM-GNAMPA for the kind support. The second author would like to thank Sapienza (Ricerca Scientifica 2020) and the group INdAM-GNAMPA for the grants supporting this research.


%

\end{document}